\documentclass[11pt]{elsarticle}
\usepackage[utf8]{inputenc}
\usepackage[T1]{fontenc}
\usepackage{graphicx}
\usepackage{grffile}
\usepackage{longtable}
\usepackage{wrapfig}
\usepackage{rotating}
\usepackage[normalem]{ulem}
\usepackage{amsmath}
\usepackage{textcomp}
\usepackage{amssymb}
\usepackage{capt-of}
\usepackage{hyperref}
\usepackage{hyperref}
\hypersetup{
pdftitle={Manuscript},
pdfauthor={Xiaoyu Wei},
pdfprintscaling=None}
\usepackage[l2tabu,orthodox]{nag}
\usepackage[english]{babel}
\usepackage{microtype}
\usepackage{graphicx}
\usepackage{tikz}
\usepackage{subcaption}
\newcommand{\pvint}{\mathop{\mathrm PV}\!\int}

\usepackage{booktabs}
\usepackage{multirow}
\usepackage[capitalise]{cleveref}
\usepackage{listings}
\usepackage{amsthm}
\theoremstyle{plain}
\newtheorem{theorem}{Theorem}

\newtheorem{proposition}[theorem]{Proposition}
\hypersetup{
 pdfauthor={},
 pdftitle={},
 pdfkeywords={},
 pdfsubject={},
 pdfcreator={Emacs 27.2 (Org mode 9.3.6)},
 pdflang={English}}
\usepackage{tcolorbox}
\newcommand\reviewresponse[1]{%
  \ifthenelse{\isundefined{\DIFaddtex}}%
  {}%
  {\marginpar{\begin{tcolorbox}[left=1mm,right=1mm]\tiny #1\end{tcolorbox}}}%
}

\begin{document}

\title{Integral Equation Methods for the Morse-Ingard~Equations}

\author[myaddress]{Xiaoyu Wei}
\ead{xywei@illinois.edu}

\author[myaddress]{Andreas Kl{\"o}ckner\corref{mycorrespondingauthor}}
\cortext[mycorrespondingauthor]{Corresponding author}
\ead{andreask@illinois.edu}

\author[bayloraddress]{Robert C.~Kirby}
\ead{Robert\_Kirby@baylor.edu}

\address[myaddress]{Department of Computer Science, University of Illinois at
Urbana-Champaign,\\ Champaign, Illinois, United States}

\address[bayloraddress]{Department of Mathematics, Baylor University, Waco, Texas, United States}

\begin{abstract}
    We present two (a decoupled and a coupled) integral-equation-based methods
    for the Morse-Ingard equations subject to Neumann boundary conditions on
    the exterior domain.  Both methods are based on second-kind integral
    equation (SKIE) formulations. The coupled method is well-conditioned and
    can achieve high accuracy. The decoupled method has lower computational
    cost and more flexibility in dealing with the boundary layer; however, it
    is prone to the ill-conditioning of the decoupling transform and cannot
    achieve as high accuracy as the coupled method. We show numerical examples
    using a Nyström method based on quadrature-by-expansion (QBX) with
    fast-multipole acceleration.  We demonstrate the accuracy and efficiency of
    the solvers in both two and three dimensions with complex geometry.
\end{abstract}


\maketitle

\section{Introduction}
\label{sec:introduction}

The Morse-Ingard equations are time-harmonic, steady-state equations derived
from the linearized Navier-Stokes equations \cite{kaderli_analytic_2017, morse_theoretical_1987}. They model the pressure and the temperature variations
of a fluid due to a heat source inside the fluid, in a regime where both the
acoustic wavelength and the thermal boundary layer thickness are of interest.
Specifically, we consider the following nondimensionalized Morse-Ingard
equations from \cite{kaderli_analytic_2017}: in the fluid domain $D^c$ with
\(D\subset\mathbb{R}^d\) bounded,
\begin{equation}
\left\{
\begin{aligned}
    & \Omega \nabla^2 T + i T - i \frac{\gamma - 1}{\gamma} P = S,\\
& (1 - i \gamma \Lambda) \nabla^2 P +
  \left[ \gamma\left( 1 - \frac{\Lambda}{\Omega}\right) + \frac{\Lambda}{\Omega} \right] P -
  \gamma\left( 1- \frac{\Lambda}{\Omega}\right) T = - i
  \gamma \frac{\Lambda}{\Omega} S,
\end{aligned}
\right.
\label{eqn:morse-ingard-pde}
\end{equation}
where $P$ is the pressure field, $T$ is the temperature field,
$S$ is the heat source, and $\Omega,\Lambda,\gamma$ are dimensionless
parameters.

The application that motivated this work is to model trace gas sensors
that utilize optothermal and photoacoustic and effects to aid designing such
sensors. The {\bf q}uartz-{\bf e}nhanced {\bf p}hoto{\bf a}coustic {\bf s}pectroscopy
(QEPAS) sensor, for example, employs a quartz tuning fork to detect
via the piezoelectric effect the acoustic pressure waves that are generated
when optical radiation from a laser is periodically absorbed by molecules of a
trace gas \cite{kosterev_quartzenhanced_2002, kosterev_applications_2005}. The
{\bf r}esonant {\bf o}pto{\bf t}hermo{\bf a}coustic {\bf de}tection (ROTADE)
sensor, on the other hand, uses the same tuning fork to detect the thermal
diffusion wave via the indirect pyroelectric effect
\cite{kosterev_resonant_2010}.
An efficient and accurate solver for the Morse-Ingard equations on the exterior
domain can be a key tool for the modeling and design of both QEPAS and ROTADE
sensors \cite{kaderli_analytic_2017,petra_modeling_2011,safin_modeling_2018}.

To suit this application, we consider
\reviewresponse{Reviewer 2, comment 1}
the thermoacoustic scattering problem, where
the {\it scattered waves} obey the homogeneous Morse-Ingard equations on
the exterior domain (\(S=0\)), coupled with sound-hard boundary conditions in
pressure and continuity in heat flux, leading to Neumann boundary conditions in
both $T$ and $P$,
\begin{equation}
\left\{
\begin{aligned}
& \left.\frac{\partial T}{\partial n}\right\rvert_{\partial D} = g_T, \\
& \left.\frac{\partial P}{\partial n}\right\rvert_{\partial D} = g_P,
\end{aligned}
\right.
\label{eqn:morse-ingard-bc-omega}
\end{equation}
where the volumetric source $S$ in \eqref{eqn:morse-ingard-pde}
is accounted for by the {\it incoming waves} it induces
in the boundary conditions above.

For solution uniqueness, we also require
\begin{equation}
T(x) < \infty \text{ and } P(x) < \infty,
\quad \text{when } |x| \rightarrow \infty.
\label{eqn:morse-ingard-bc-infty}
\end{equation}
As shown in \cite{safin_modeling_2018}, \eqref{eqn:morse-ingard-bc-infty}
ensures solution uniqueness by ruling out unphysical waves from the infinity,
similar to what the Sommerfeld radiation condition does for the Helmholtz
equation. In our applications, the wave number has positive imaginary part.
Therefore, boundedness at infinity is sufficient. For general wave numbers, we
also provide a direct generalization of the Sommerfeld condition in Section
\ref{sec:sommerfeld}.

In this paper, we propose two integral-equation-based methods for
\labelcref{eqn:morse-ingard-pde,eqn:morse-ingard-bc-omega,eqn:morse-ingard-bc-infty}.
For both methods, we derive representations of the solution in terms of layer potentials
involving unknown densities, leading to integral equations that can be reduced to the form
of a Fredholm integral equation of the second kind
\begin{equation}
    (I - A)\rho = f,
    \label{eqn:skie0}
\end{equation}
where $A$ is a compact integral operator;
therefore, our methods are all based on second-kind integral equation (SKIE) formulations.
SKIEs are attractive because both the condition number and the number of iterations required
for iterative solvers like GMRES \cite{saad_gmres_1986a} are bounded by constant when refining the mesh.  In fact,~\cite{moret_note_1997} gives rigorous superlineary GMRES convergence estimates in this case.
The first method is based on a direct SKIE formulation to the
original equations using the free-space Green's functions.  While deriving the
analytic formula for the free-space Green's functions, we also obtain a
decoupling transform that converts the problem into two decoupled Helmholtz
equations, with decoupled Neumann data.  Our second method takes advantage of
this transform and uses SKIEs for the Helmholtz equations to obtain the
solution.  As such we refer to the first method {\it the coupled method} and
the second {\it the decoupled method}. We solve the SKIEs using a Nyström method
with GMRES, which requires evaluating layer potentials at the boundary using a
suitable singular quadrature scheme.  For our numerical experiments, we use
quadrature-by-expansion (QBX) with fast-multipole acceleration
\cite{klockner_quadrature_2013, wala_fast_2018}. Consequently, both methods have
linear complexity with respect to the number of degrees of freedoms.  It is
noteworthy that our methods are agnostic of the singular quadrature
scheme.  Our methods are also capable of using high order discretization and
can handle complex geometries. We demonstrate this through numerical examples
in two and three dimensions.


The rest of this paper is organized as follows. We first present the derivation
for the free-space Green's functions and the decoupling transform in Section
\ref{sec:modes}. Then we give the SKIE formulations for both the decoupled and
coupled methods in Section \ref{sec:skie}, and present some details of our
numerical implementation in Section \ref{sec:method}. After that, we present
the numerical results in Section \ref{sec:results}, and give some concluding
discussion in Section \ref{sec:conclusion}.

\section{Analysis of the problem}
\label{sec:modes}

\subsection{Thermal and acoustic modes}
In order to obtain an integral equation formulation, we require the free-space
Green's function for the Morse-Ingard equations (\ref{eqn:morse-ingard-pde}).
Following the derivation for the analytic solution to the Morse-Ingard
equations in a cylindrically symmetric geometry in \cite{kaderli_analytic_2017},
we first identify the eigenmodes by looking for particular solutions where \(T\)
is an eigenfunction of \(\nabla^2\), s.t. \(\nabla^2T = -k^2 T\).
Substituting into the first equation of \eqref{eqn:morse-ingard-pde} yields
\begin{equation}
P = \frac{\gamma}{\gamma - 1} [(1 + i \Omega k^2) T + i S].
\label{eqn:pressure-mode-ansatz}
\end{equation}
Consider the homogeneous case by letting \(S = 0\), then \(P = mT\) is also an
eigenfunction of \(\nabla^2\), where the constant $m = \frac{\gamma}{\gamma - 1}
(1 + i \Omega k^2)$.  Substituting \eqref{eqn:pressure-mode-ansatz} and \(P=mT\)
into \eqref{eqn:morse-ingard-pde} yields
\begin{equation}
\label{eqn:quadratic-k}
(1 - i \gamma \Lambda) (-k^2 mT) + \left[ \gamma\left( 1 -
\frac{\Lambda}{\Omega}\right) + \frac{\Lambda}{\Omega} \right] (mT) -
\gamma\left( 1- \frac{\Lambda}{\Omega}\right) T = 0.
\end{equation}
For (\ref{eqn:quadratic-k}) to have nontrivial solution, the coefficient of $T$ must vanish,
so that
\begin{equation}
(i\Omega + \gamma \Omega \Lambda)k^4 +
(1 - i \gamma \Omega - i \Lambda) k^2 - 1= 0.
\label{eqn:eigen-algebraic}
\end{equation}
Let $Q$ to be a complex constant such that $Q^2 = 4(i\Omega + \gamma \Omega
\Lambda) + (1 - i \gamma \Omega - i \Lambda)^2$.  Based on physical interpretation,
we classify the roots of \eqref{eqn:eigen-algebraic} into two groups:
\begin{enumerate}
\item \(k_t\) corresponding to the thermal modes that attenuate rapidly:
\begin{equation}
k_t^2 = \frac{i}{2\Omega} \left(\frac{1 - i\gamma\Omega - i\Lambda + Q}{1 -
i\gamma\Lambda}\right), \quad m_t := \frac{\gamma}{\gamma - 1}(1 + i\Omega
k_t^2).
\end{equation}
\item \(k_p\) corresponding to the acoustic modes that attenuate slowly:
\begin{equation}
k_p^2 = \frac{i}{2\Omega} \left(\frac{1 - i\gamma\Omega - i\Lambda - Q}{1 -
i\gamma\Lambda}\right), \quad m_p := \frac{\gamma}{\gamma - 1}(1 + i\Omega
k_p^2).
\end{equation}
\end{enumerate}

Since the eigenfunctions of \(\nabla^2\) form a basis of \(H^1\), we obtain the
fundamental set of solutions to the homogeneous problem under radial
symmetry, denoted $U_d$, $(d=2,3)$,
\begin{equation*}
  \begin{bmatrix} T(r) \\ P(r) \end{bmatrix} \in U_d.
\end{equation*}
In two dimensions (\(d=2\)),
\begin{equation}
    U_2 = \operatorname{span}\left\{
    \begin{bmatrix} J_0(k_pr) \\ m_pJ_0(k_pr) \end{bmatrix},
    \begin{bmatrix} H^{(1)}_0(k_pr) \\ m_pH^{(1)}_0(k_pr) \end{bmatrix},
    \begin{bmatrix} J_0(k_tr) \\ m_tJ_0(k_tr) \end{bmatrix},
    \begin{bmatrix} H^{(1)}_0(k_tr) \\ m_tH^{(1)}_0(k_tr) \end{bmatrix}
    \right\},
\end{equation}
where \(J_0\) is the Bessel function of the first kind of order zero,
\(H^{(1)}_0\) is the Hankel function of the kind of order zero.
Similarly, in three dimensions (\(d=3\)),
\begin{equation}
    U_3 = \operatorname{span}\left\{
    \begin{bmatrix} j_0(k_pr) \\ m_pj_0(k_pr) \end{bmatrix},
    \begin{bmatrix} h^{(1)}_0(k_pr) \\ m_ph^{(1)}_0(k_pr) \end{bmatrix},
    \begin{bmatrix} j_0(k_tr) \\ m_tj_0(k_tr) \end{bmatrix},
    \begin{bmatrix} h^{(1)}_0(k_tr) \\ m_th^{(1)}_0(k_tr) \end{bmatrix} \right\},
\end{equation}
where \(j_0\) is the spherical Bessel function of the first kind of order zero,
\(h^{(1)}_0\) is the spherical Bessel function of the third kind of order zero.
We also note that \(h_0^{(1)}\) has simple closed form
\begin{equation}
h_0^{(1)}(r) = j_0(r) + iy_0(r) = \frac{\sin r}{r} - i \frac{\cos r}{r}
= \frac{-i}{r} e^{ir}.
\end{equation}
Note that the choice of basis is not unique. We deliberately chose the above
basis functions so that the condition in \eqref{eqn:morse-ingard-bc-infty} can be
easily enforced.

\subsection{The decoupled equations}
\label{sec:decouple}
From the above radial symmetry solutions, it is obvious that the Morse-Ingard
equations are a linear superposition of two Helmholtz-type equations. To get
the actual change of variables that decouples the PDE system, we solve for
\(t\in \mathbb{C}\) such that the sum of the first equation and \(t\) times the
second equation in \eqref{eqn:morse-ingard-pde} reduces to a scalar
Helmholtz-type PDE
\begin{equation}
a_1(t) \nabla^2 T + a_2(t) \nabla^2 P + a_3(t) T + a_4(t) P = a_5(t) S,
\label{eqn:decoupling-equation}
\end{equation}
where \(a_1=\Omega\), \(a_2=(1-i\gamma\Lambda)t\),
\(a_3=i - \gamma\left( 1- \frac{\Lambda}{\Omega}\right) t\),
\(a_4=-i\frac{\gamma-1}{\gamma} +
\left[
 \gamma\left( 1 - \frac{\Lambda}{\Omega}\right) + \frac{\Lambda}{\Omega}
\right]t\),
\(a_5=1 - i \gamma \frac{\Lambda}{\Omega} t\) are all linear functions of \(t\). The
condition under which \eqref{eqn:decoupling-equation} becomes a scalar
Helmholtz-type PDE is \(a_1a_4 = a_2a_3\), which is a quadratic equation of
\(t\) and admits two roots
\begin{equation}
\begin{aligned}
t_\pm
& = \frac{
 (2\Lambda\gamma-\Lambda-\Omega\gamma+i)\Omega
 \mp i\Omega Q }{
 2\gamma(\Lambda-\Omega)(i\Lambda\gamma-1)}.
\end{aligned}
\label{eqn:decoupling-roots}
\end{equation}
Letting \(V_t = \Omega T + t_+(1-i\gamma\Lambda)P\),
\(V_p = \Omega T + t_-(1-i\gamma\Lambda)P\),
we find the decoupled scalar PDEs,
\begin{equation}
\begin{aligned}
\nabla^2V_t + k_t^2 V_t &= a_5(t_+)S, \\
\nabla^2V_p + k_p^2 V_p &= a_5(t_-)S, \\
\end{aligned}
\label{eqn:decoupled-pdes}
\end{equation}
as expected from the thermal and acoustic modes.

\subsection{Sommerfeld radiation condition}
\label{sec:sommerfeld}

Now we consider the boundary conditions for the decoupled equations (\ref{eqn:decoupled-pdes}).
The boundary conditions on $\partial D$ are still decoupled Neumann,
\begin{equation}
\begin{aligned}
    & \left.\frac{\partial V_t}{\partial n}\right\rvert_{\partial D} = \Omega g_T + t_+(1-i\gamma\Lambda)g_P, \\
    & \left.\frac{\partial V_p}{\partial n}\right\rvert_{\partial D} = \Omega g_T + t_-(1-i\gamma\Lambda)g_P.
\end{aligned}
\label{eqn:neumann-decoupled}
\end{equation}
For the far-field conditions, we can either impose boundedness of $V_t, V_p$ at infinity, or apply the classical
far-field conditions for the Helmholtz equation due to Arnold Sommerfeld,
\begin{equation}
\begin{aligned}
\lim_{|x|\rightarrow\infty}|x|^{\frac{d-1}{2}}\left(
 \frac{\partial}{\partial|x|} - ik_t\right)V_t &= 0, \\
\lim_{|x|\rightarrow\infty}|x|^{\frac{d-1}{2}}\left(
 \frac{\partial}{\partial|x|} - ik_p\right)V_p &= 0. \\
\end{aligned}
\label{eqn:sommerfeld-decoupled}
\end{equation}

\subsection{Free-space Green's function}
\label{sec:greens-func}

By letting \(S=\delta(r)\) in \eqref{eqn:morse-ingard-pde}, the free-space
Green's function satisfies the following equations in the weak sense,
\begin{equation}
\left\{
\begin{aligned}
  & \Omega \nabla^2 T + i T - i \frac{\gamma - 1}{\gamma} P = \delta, \\
  & (1 - i \gamma \Lambda) \nabla^2 P +
  \left[ \gamma\left( 1 - \frac{\Lambda}{\Omega}\right) +
  \frac{\Lambda}{\Omega} \right] P -
  \gamma\left( 1- \frac{\Lambda}{\Omega}\right) T = - i \gamma
  \frac{\Lambda}{\Omega} \delta.
\end{aligned}
\right.
\label{eqn:greens-weak-pde}
\end{equation}

We first seek for weak solutions \([T(r), P(r)]^T \in U_d\) that also satisfy
the far-field conditions in \eqref{eqn:morse-ingard-bc-infty}.  In two
dimensions, because when $k$ has positive imaginary part, solutions with
nonzero \(J_0\) components are unbounded at infinity, and represent a wave
traveling from infinity towards \(0\), enforcing the far-field conditions
amounts to restricting the solution to the two-dimensional subspace of $U_2$
spanned by bases involving \(H_0^{(1)}\).  Similarly, we restrict the solution
to the two-dimensional subspace of $U_3$ spanned by bases involving \(h_0^{(1)}
\) in \(3\)D.

We note that, in the weak sense, $\nabla^2_{2D} \left(\frac{1}{2\pi} \ln r\right)
= \nabla^2_{3D} \left(- \frac{1}{4\pi} \frac{1}{r}\right) = \delta(r)$.
\reviewresponse{Reviewer 2, comment 2}
The following asymptotic limits hold when \(r\rightarrow 0\) (\cite{olver_nist_2010}),
\begin{align}
    & \nabla^2_{2D} H_0^{(1)}(kr) \sim\nabla^2 [(2i/\pi) \ln (kr)] = 4i \delta(r), \\
    & \nabla^2_{3D} h_0^{(1)}(kr) \sim\nabla^2 [- i/(kr)] = (4\pi i /k)\delta(r).
\end{align}
When \(r\neq 0\), any linear combination in the fundamental set $U_d$ solves
the system. All we need is to match the leading order terms in the neighborhood
of \(0\). Since we have restricted the solutions to a two-dimensional space,
matching the two coefficients on the right hand side of (\ref{eqn:greens-weak-pde})
uniquely determines the Green's function.

\subsubsection{\(2\)D}
\label{sec:org65336f4}

In two dimensions, we let
\begin{equation}
\begin{aligned}
& G_T = b_1 H_0^{(1)}(k_p r) + b_2 H_0^{(1)}(k_t r), \\
& G_P = b_1 m_p H_0^{(1)}(k_p r) + b_2 m_t H_0^{(1)}(k_t r).
\end{aligned}
\end{equation}
Substituting into \eqref{eqn:greens-weak-pde} and matching the leading order
terms yields
\begin{equation}
\begin{aligned}
& 4i\Omega(b_1 + b_2)\delta = \delta, \\
& 4i(1-i\gamma\Lambda)(b_1 m_p + b_2 m_t)\delta =
  -i\gamma\frac{\Lambda}{\Omega}\delta.
\end{aligned}
\end{equation}
Solving this linear system, we have
\begin{equation}
\begin{aligned}
& b_1 = \frac{\gamma-1}{4\gamma\Omega Q}[im_t + (m_t-1)\gamma\Lambda], \\
& b_2 = - \frac{\gamma-1}{4\gamma\Omega Q}[im_p + (m_p-1)\gamma\Lambda].
\end{aligned}
\end{equation}

\subsubsection{\(3\)D}
\label{sec:org7d34b1e}
Similarly, in three dimensions, we let
\begin{equation}
\begin{aligned}
& G_T = c_1 h_0^{(1)}(k_p r) + c_2 h_0^{(1)}(k_t r), \\
& G_P = c_1 m_p h_0^{(1)}(k_p r) + c_2 m_t h_0^{(1)}(k_t r).
\end{aligned}
\end{equation}
Substituting into \eqref{eqn:greens-weak-pde} and matching the leading order
terms yields
\begin{equation}
\begin{aligned}
& 4\pi i \Omega\left(\frac{c_1}{k_p} + \frac{c_2}{k_t}\right)\delta = \delta \\
& 4\pi i(1-i\gamma\Lambda) \left(\frac{c_1 m_p}{k_p} + \frac{c_2 m_t}{k_t}\right)\delta
  = -i\gamma\frac{\Lambda}{\Omega}\delta.
\end{aligned}
\end{equation}
The solution is
\begin{equation}
\begin{aligned}
& c_1 = \frac{k_p(\gamma-1)}{4\pi\gamma\Omega Q}[im_t + (m_t-1)\gamma\Lambda], \\
& c_2 = - \frac{k_t(\gamma-1)}{4\pi\gamma\Omega Q} [im_p +(m_p - 1)\gamma\Lambda].
\end{aligned}
\end{equation}

\subsection{Green's representation formula}
\label{sec:greens-formula}

Since the solution to the Helmholtz equation satisfies Green's representation formula,
we expect a similar relation holds for solution to the Morse-Ingard equations as well.
Let \(G_p, G_t\) be the Helmholtz kernels of wave number \(k_p, k_t\),
respectively.
Given a density $\rho\in C(\partial D)$, define the {\it single-layer potential operators}
\begin{equation}
\begin{aligned}
    S_t[\rho](x)= \int_{\partial D} G_t(x,y)\rho(y) ds_y, \\
    S_p[\rho](x)= \int_{\partial D} G_p(x,y)\rho(y) ds_y,
\end{aligned}
\end{equation}
and the {\it double-layer potential operators}
\begin{equation}
\begin{aligned}
    D_t[\rho](x)= \int_{\partial D} \partial_{n_y}G_t(x,y)\rho(y) ds_y, \\
    D_p[\rho](x)= \int_{\partial D} \partial_{n_y}G_p(x,y)\rho(y) ds_y.
\end{aligned}
\end{equation}
\reviewresponse{Reviewer 2, comment 3}
If $x\in\partial D$, the integrals above are to be understood in the principal-value sense.

Let \(D\) be a bounded region in \(\mathbb{R}^n\), and \(\partial D\) is piecewise
\(C^2\). Then from Green's representation formula for the Helmholtz
equation \cite[Theorem 12]{evans_partial_2010}, for all \(x\in D\),
\begin{equation}
\begin{aligned}
    D_t[V_t] - S_t[\partial_nV_t] = -V_t,\\
    D_p[V_p] - S_p[\partial_nV_p] = -V_p.
\end{aligned}
\label{eqn:greens-formula-decoupled}
\end{equation}
Denote the left hand sides of (\ref{eqn:greens-formula-decoupled}) as $L_t[V_t]$ and
$L_p[V_p]$, respectively.
Recalling that \(V_{t,p} = \Omega T + t_\pm(1-i\gamma\Lambda)P\), we have
\begin{equation}
\begin{aligned}
    \Omega L_t[T] + t_+(1-i\gamma\Lambda) L_t[P] = -\Omega T - t_+(1-\gamma\Lambda)P, \\
    \Omega L_p[T] + t_-(1-i\gamma\Lambda) L_p[P] = -\Omega T - t_-(1-\gamma\Lambda)P,
\end{aligned}
\label{eqn:greens-formula-coupled}
\end{equation}
inside \(D\). \eqref{eqn:greens-formula-coupled} gives representations of \(P\) and \(T\) in terms of
layer potentials.
Lastly, we note that since all layer potentials above are zero at infinity, the assumption that
$D$ is bounded can be removed and (\ref{eqn:greens-formula-coupled}) still holds.

\section{SKIE formulations}
\label{sec:skie}

\subsection{The decoupled method}

As demonstrated above, the Morse-Ingard equations can be decoupled into two Helmholtz-like PDEs
that have decoupled Neumann boundary conditions on $\partial D$. The most straightforward scheme
is to solve the decoupled PDEs separately and recombine the results. Both
decoupled equations have the form
\reviewresponse{Reviewer 2, comment 2}
\begin{equation}
    \label{eqn:decoupled-skie}
    \left\{
    \begin{aligned}
        & \nabla^2V + k^2 V = 0, \quad \text{in } D^c,\\
        & \partial_n V = h, \quad \text{on } \partial D,
    \end{aligned}
    \right.
\end{equation}
where $V \in \{V_t, V_p\}$, $k \in \{k_t, k_p\}$, $h=\Omega g_T + t_\pm(1-i\gamma\Lambda) g_P$, respectively.

We solve (\ref{eqn:decoupled-skie}) using a single-layer potential representation. For $x\in\partial D$, let
\reviewresponse{Reviewer 2, comment 3}
\begin{equation}
    \begin{aligned}
        & V_t(x) = S_t[\sigma_t](x) = \pvint_{\partial D} \overline{G}(k_t|x-y|) \sigma_t(y) ds_y, \\
        & V_p(x) = S_p[\sigma_p](x) = \pvint_{\partial D} \overline{G}(k_p|x-y|) \sigma_p(y) ds_y, \\
    \end{aligned}
\end{equation}
where $\sigma_{t,p} \in L^2(\partial D)$, $\overline{G}(r) = \frac{1}{2\pi} \ln r$ in 2\(D\)
and $\overline{G}(r) = -\frac{1}{4\pi r}$ in 3\(D\). This leads to the following SKIE for each decoupled component
\cite{colton_integral_2013a},
\begin{equation}
    \label{eqn:skie-components}
    - \frac{1}{2} \sigma (x) + \pvint_{\partial D} \partial_n \overline{G}(k|x-y|) \sigma(y) ds_y  = h(x), \quad \text{on } \partial D.
\end{equation}

After obtaining $V_t, V_p$, we then recover $T, P$ by applying the inverse change of variables
\begin{equation}
    \begin{bmatrix} T \\ P \end{bmatrix} =
        \begin{bmatrix} \Omega & t_+(1-i\gamma\Lambda) \\ \Omega & t_-(1-i\gamma\Lambda) \end{bmatrix}^{-1}
            \begin{bmatrix} V_t \\ V_p \end{bmatrix}.
\end{equation}

\subsection{The coupled method}

Alternatively, we can also use the Green's function for the Morse-Ingard equations to directly formulate
a coupled SKIE. Denote the free-space Green's functions in \(2\)/\(3\)D by
\begin{equation}
G(r) = \begin{bmatrix}G_T(r) \\ G_P(r)\end{bmatrix}.
\end{equation}
We assume that \(k_t^2 - k_p^2 = -\frac{Q}{\Omega(i + \gamma\Lambda)} \neq 0\), which is well-justified on physical grounds (the thermal and acoustic waves have fundamentally different properties).  Then,
\begin{equation}
 G_1(r) = - \frac{\alpha_1\Omega(1 - i\gamma\Lambda)}{4b_2Q}
          (\nabla^2 + k_p^2) G(r),
\end{equation}
where \(\alpha_1 = 1\) in \(2\)D, and \(\alpha_1 = \frac{k_t}{\pi}\) in \(3\)D.
And let
\begin{equation}
G_2(r) = \frac{\alpha_2\Omega(1 - i\gamma\Lambda)}{4b_1Q}
         (\nabla^2 + k_t^2) G(r),
\end{equation}
where \(\alpha_2 = 1\) in \(2\)D, and \(\alpha_2 = \frac{k_p}{\pi}\) in \(3\)D.

Obviously, \(G_1, G_2\) satisfy the homogeneous PDE when \(r\neq 0\) because they
consist of \(G\) and its derivatives. In fact, they are constructed to have very
simple explicit formulae:
\begin{equation}
G_1(r) = - \frac{1}{4i}\begin{bmatrix}
    H_0^{(1)}(k_tr) \\ m_tH_0^{(1)}(k_tr) \\ \end{bmatrix}, \quad
G_2(r) = - \frac{1}{4i} \begin{bmatrix}
    H_0^{(1)}(k_pr) \\ m_pH_0^{(1)}(k_pr) \\ \end{bmatrix} \quad
\text{in \(2\)D},
\end{equation}
and
\begin{equation}
G_1(r) = - \frac{k_t}{4\pi i} \begin{bmatrix}
    h_0^{(1)}(k_tr) \\ m_th_0^{(1)}(k_tr) \\ \end{bmatrix}, \quad
G_2(r) = - \frac{k_p}{4\pi i} \begin{bmatrix}
    h_0^{(1)}(k_pr) \\ m_ph_0^{(1)}(k_pr) \\ \end{bmatrix} \quad
\text{in \(3\)D}.
\end{equation}

To account for the two boundary conditions in \eqref{eqn:morse-ingard-bc-omega},
we need two scalar densities. Let \(\sigma = [\sigma_1, \sigma_2]^T\), \(\sigma_i
\in L^2(\partial D)\). Define the (vector-valued) single-layer potentials of a scalar density
\(\rho\) to be
\begin{equation}
    S_i[\rho](x) = \pvint_{\partial D} G_i(|x-y|) \rho(y) ds_y, \quad x\in\partial D,\text{ and } i=1,2.
\end{equation}
We consider the following solution representation
\begin{equation}
u = S_1[\sigma_1] + S_2[\sigma_2].
\label{eqn:skie-representation}
\end{equation}
We claim that \eqref{eqn:skie-representation} gives an SKIE for the Morse-Ingard
equation. To show that, we first present the jump relations of the layer
potential operators used in the construction.

\begin{proposition}[Jump relations]
\(S_i[\sigma_i]\) is continuous across \(\partial D\), and its normal derivative
satisfies the following jump relations: for \(z\in\partial D\),
\begin{equation}
\frac{\partial}{\partial n_z}S_i(z)_{\pm} := \lim_{x\rightarrow z\pm}
  \frac{\partial}{\partial n_z}S_i[\sigma_i](x) = \frac{\partial}{\partial
  n_z}S_i[\sigma_i](z) \mp \frac{1}{2} \begin{bmatrix} c_i \\ d_i \end{bmatrix}
  \sigma_i(z),
\end{equation}
where
\begin{equation}
\begin{aligned}
& c_1 = c_2 = 1,\quad d_1 = m_t,\quad d_2 = m_p \quad\text{in \(2\)D}, \\
& c_1 = \frac{1}{k_t},\quad c_2 = \frac{1}{k_p}\quad,
  d_1 = \frac{m_t}{k_t},\quad d_2 = \frac{m_p}{k_p} \quad\text{in \(3\)D}.
\end{aligned}
\end{equation}
\label{prop:jump-relations}
\end{proposition}
Since in our construction, the layer potential operators are standard
single-layer potentials of the Helmholtz equation (up to constant
multiplications), the proof of \eqref{prop:jump-relations} follows the same exact
steps as in \cite[Chapter~2.4-2.5]{colton_integral_2013a} for the Helmholtz
equation.

Substituting the representation of \(u\) into \eqref{eqn:morse-ingard-bc-omega} yields
the following boundary integral equation
\begin{equation}
\left(- \frac{1}{2} \begin{bmatrix} c_1 & c_2 \\ d_1 & d_2
\end{bmatrix} + \begin{bmatrix} \partial_nS_{1T} & \partial_nS_{2T}
\\ \partial_nS_{1P} & \partial_nS_{2P} \end{bmatrix} \right) \begin{bmatrix}
\sigma_1\\\sigma_2 \end{bmatrix}(z)
= \begin{bmatrix} \partial_n T \\ \partial_n P \end{bmatrix}(z),
\quad \forall z \in \partial D.
\label{eqn:skie}
\end{equation}

Note that under the current assumption of \(k_t^2 \neq k_p^2\), we have
\begin{equation}
\det\left(\begin{bmatrix} c_1 & c_2 \\ d_1 & d_2 \end{bmatrix}\right)
  = c_1d_2 - c_2d_1 \neq 0.
\end{equation}

To supply smoothness for our proofs of numerical accuracy, we note
a compactness result for the operators in \eqref{eqn:skie}.
\begin{proposition}
  \label{thm:compact}
  If $\partial D\in C^{2}$, for $0<\alpha < 1$,
  the operator
  \[
  L:
  \begin{bmatrix}
  \sigma_1 \\
  \sigma_2
  \end{bmatrix}
  \mapsto
  \begin{bmatrix}
  \partial_nS_{1T} & \partial_nS_{2T}\\
  \partial_nS_{1P} & \partial_nS_{2P}
  \end{bmatrix}
  \begin{bmatrix}
  \sigma_1 \\
  \sigma_2
  \end{bmatrix}
  \]
    is a compact operator from $C^{0,\alpha}\times C^{0,\alpha}$ into itself.
\end{proposition}
\begin{proof}
    Since $L$ is a $2\times 2$ operator matrix, we only need to show that each entry is compact.
    Being adjoint operators of the double layer potentials (i.e., derivatives of rescaled Helmholtz single layer potentials),
    \(\partial_n S_{i*}\) are compact operators from Hölder space $C^{0,\alpha}(\partial D)$
    into itself ({\cite[Theorem 7.5]{kress_linear_2012}}).
Therefore, \eqref{eqn:skie} is an SKIE.
\end{proof}

Regarding analytical properties of the operator, the most relevant results we
found in the literature are for layer potentials with Laplace kernels,
which admit the following regularity estimates:
\begin{proposition}[Layer potential regularity {\cite[Prop.~4.2]{mazya_higher_2005}}]
    \label{thm:mazya}
    Let $d\geq 2$, $p\in(1,\infty)$, and $p(\ell-1)>d-1$, where $\ell$ is a
    noninteger, $\ell > 1$. And let $\alpha = 1-\{\ell\}-1/p$. Suppose that
    $\partial D$ is connected and satisfies the Sobolev graph property
    $\partial D \in W^\ell_p$, that is, for every point $O\in\partial D$ there
    exists a neighbourhood $U$ and $f \in W^\ell_p(\mathbb{R}^{d-1})$ such that
    \[
        U\cap\Omega = U\cap\{(x,y) | x\in\mathbb{R}^{d-1}, y>f(x)\}.
    \]
    Then
    \[
        \begin{aligned}
            & \|(S\rho)_+\|_{W^{\lfloor\ell\rfloor + 1, \alpha}_p(D)} \leq c    \|\rho\|_{W^{\ell-1}_p(\partial D)}, \\
            & \|(S\rho)_-\|_{W^{\lfloor\ell\rfloor + 1, \alpha}_p(D)} \leq c(D) \|\rho\|_{W^{\ell-1}_p(\partial D)}, \\
        \end{aligned}
    \]
    where the weighted Sobolev norm is defined as
    \[
        \|u\|_{W^{m, \alpha}_p(D)} =  \left( \int_D (\operatorname{dist}(z,\partial D))^{p\alpha} |\nabla_m u(z)|^p dz \right)^{1/p} + \|u\|_{L_p(D)},
    \]
    and $(S\rho)_\pm$ are the limits of the Laplace single layer potentials
    from the outside and the inside, respectively.
    In particular, for the adjoint operator,
    \[
        \|D^\star\rho\|_{W^{\ell - 1}_p(\partial D)} \leq c \|\rho\|_{W^{\ell-1}_p(\partial D)}.
    \]
\end{proposition}

\section{Numerical implementation and error analysis}
\label{sec:method}

For the coupled method, the left hand side operator of the SKIE is a compact perturbation
to a block-constant operator. This structure inspires us to try to use the
inverse of the constant part in \eqref{eqn:skie} as a left block preconditioner
\begin{equation}
P = \left(-\frac{1}{2} \begin{bmatrix} c_1 & c_2 \\ d_1 & d_2
 \end{bmatrix}\right)^{-1}.
\end{equation}
The idea is to make the preconditioned operator ``almost'' block-diagonal in the sense
\reviewresponse{Reviewer 1, comment 1}
that its spectrum only has one cluster point in the complex plane, as will be
shown in the numerical tests.
Doing so empirically reduces the number of GMRES iterations needed by up-to one half.

As will be shown in the next section, for the realistic parameters used in our
tests, the thermal modes decay rapidly relative to the acoustic modes, leading
to two widely separated spatial scales. This may lead to numerical difficulties
for the coupled method.  Due to the fast decay of thermal modes, however, if we
are only interested in the solution in the bulk region away from the boundary,
we can solve only for $V_p$ in (\ref{eqn:decoupled-skie}) using a single-layer
potential representation
\reviewresponse{Reviewer 2, comment 3}
\begin{equation}
       V_p(x) = S_p[\sigma_p](x) = \pvint_{\partial D} \overline{G}(k_p|x-y|) \sigma_p(y) ds_y,
       \quad x\in\partial D,
\end{equation}
and then obtain $T$ and $P$ by fixing thermal mode to be zero and applying the
inverse change of variables
\begin{equation}
    \begin{bmatrix} T \\ P \end{bmatrix} =
        \begin{bmatrix} \Omega & t_+(1-i\gamma\Lambda) \\ \Omega & t_-(1-i\gamma\Lambda) \end{bmatrix}^{-1}
            \begin{bmatrix} 0 \\ V_p \end{bmatrix}.
    \label{eqn:projection-method}
\end{equation}
Since doing so equates to projecting the problem to include only the acoustic modes,
we will refer to this modified version of the decoupled method as {\it the projection method}.

The recasting of the Morse-Ingard system \eqref{eqn:morse-ingard-pde}
in second-kind integral equations presented thus far is independent of
the specific means by which it is discretized.

We discretize the SKIEs using the Nyström's approach, that is, we discretize
\eqref{eqn:skie} by replacing the (singular) integrals in the layer potential
operators by appropriate quadrature schemes. To evaluate the singular
quadratures efficiently, we use the GIGAQBX algorithm, a
quadrature-by-expansion method with FMM acceleration
\cite{klockner_quadrature_2013,rachh_fast_2017,wala_fast_2018}.
With efficient algorithms for the left hand side operators, we solve
the discretized linear systems using GMRES.

Loosely, for Nyström discretizations of second-kind systems, Anselone's theorem
(e.g.~\cite[Theorem 10.12]{kress_linear_2012}) states that the error in
the computed density is bounded by the sum of the discretization error
in the right-hand side and the quadrature error in the evaluation of
the layer potential operators. Given our use of QBX discretization,
our quadrature error behaves as described by the following theorem:

\begin{theorem}[{QBX error estimate \citep[Thm.~1]{klockner_quadrature_2013}}]
    \label{thm:qbx-accuracy}
    Suppose that $\partial D$ is a smooth, bounded curve embedded in $\mathbb{R}^2$, that
    $B_r(c)$ is the ball of radius $r$ about $c$, and that $\overline{B_r(c)}\cap\partial D = \{x\}$.
    Let $\partial D$ be divided into $M$ panels, each of length $h$ and let $p,q$ be non-negative
    integers that define the QBX order and the number of nodes of the smooth
    Gaussian quadrature used to compute the QBX coefficients, respectively.
    For $0< \beta < 1$, there are constants $C_{p,\beta}$ and
    $\tilde{C}_{p,q,\beta}$ so that if $\phi$ lies in the Hölder space
    $C^{p,\beta}(\partial D)\cap C^{q,\beta}(\partial D)$, then the components of QBX discretization error admit the following bounds
    \begin{align}
        \label{eq:qbx-err-trunc}
        \|E_{\text{QBX,trunc}}(\phi)\|_\infty &\leq
        C_{\partial D,p,\beta} r^{p+1}\|\phi\|_{C^{p,\beta}(\partial D)},\\
        \label{eq:qbx-err-quad}
        \|E_{\text{QBX,quad}}(\phi)\|_\infty &\leq
        \tilde{C}_{\partial D,p,q,\beta} \left(\frac{h}{4r}\right)^{q}\|\phi\|_{C^{q,\beta}(\partial D)}.
    \end{align}
\end{theorem}

The use of layer potentials (such as $\partial_n S$) with more derivatives as
well as the use of FMM approximation incurs additional (controlled) error. See, for
example, \cite{epstein_convergence_2013} as well as the FMM error
analysis in \cite{wala_fast_2019} for details.

\section{Numerical results}
\label{sec:results}

In this section, we present several numerical examples that demonstrate the
accuracy and efficiency of our method.
The dimensionless parameter values used for all our tests are listed in Table
\ref{tab:param}, which are derived from the physical parameters from
\cite{kaderli_analytic_2017}. This set of parameters assumes that the fluid is
nitrogen gas at a temperature of $300$ K and a pressure of $1$ bar. It is
noteworthy that the parameter values we use closely correspond to the modeling
of laboratory experiments of trace gas sensors \cite{petra_modeling_2011}.

\begin{table}[h!]
    \centering
    \begin{tabular}{c|c}
        Symbol & Value \\
        \hline
        $\Omega$  & $3.664152973215096\times 10^{-5}$ \\
        $\gamma$  & $1.399999976158142$ \\
        $\Lambda$ & $5.370572762330994\times 10^{-5}$ \\
    \end{tabular}
    \caption{Dimensionless parameters.}
    \label{tab:param}
\end{table}

Using these dimensionless parameters to solve (\ref{eqn:quadratic-k}) yields the two wavenumbers
\begin{equation*}
    \begin{aligned}
        k_t &\approx 116.81+116.81i,\\
        k_p &\approx  1+3.42\times 10^{-5}i.
    \end{aligned}
\end{equation*}
We note that $k_t$ is the wavenumber corresponding to the fast-decaying thermal
modes, which gives rise to a boundary layer with thickness at the scale of
$\frac{1}{|\operatorname{Im}k_t|} \approx 0.01$.  On the other hand, $k_p$
corresponds to the slowly-decaying acoustic modes that have characteristic
length at the scale of $1$.

In all our tests, the GMRES tolerance is set to $10^{-14}$. We denote the computed
temperature field $T_h$ and the pressure field $P_h$, respectively, and let
\begin{equation}
    E_T(x) = \frac{|T_h(x) - T(x)|}{|T(x)|}, \quad
    E_P(x) = \frac{|P_h(x) - P(x)|}{|P(x)|}.
\end{equation}

\subsection{The coupled method and preconditioning}
\label{sec:res-coupled}

We perform convergence tests on a two-dimensional circular geometry
with radius $3.5$ centered at $(5.25, 5.25)$.
The manufactured solution is obtained by evaluating a point potential
using the free-space Green's function in Section \ref{sec:greens-func}
from sources inside the circle. We use piecewise polynomials up to order $p$ to
discretize over each panel.
For each test, we measure the error in the
infinity norm on the boundary discretization points as well as some volume
points in $[0, 10.5]^2\mathbin{\backslash} D$.

We first perform \(p\)-convergence tests by splitting the circle into $100$
equal-sized panels and increase $p$.  For this test, we set the values for the
FMM order and the QBX order to also be $p$ for simplicity. Note that both the
FMM order and the QBX order are linear functions of the number of accurate
digits, and setting them to be equal is by no means the optimal choice, but
does not affect the asymptotic convergence behavior.  Results for the coupled
method are shown in Figure \ref{fig:conv1}. In the first place, we see similar
convergence curves for different ways to measure the error norm, while the
error data measured over the boundary is less noisy. Therefore, from now on,
we will only measure the error over the boundary.  We also see similar
convergence curves when we solve with or without the block preconditioner.

\begin{figure}[ht]
  \centering
  \includegraphics[width=0.45\textwidth]{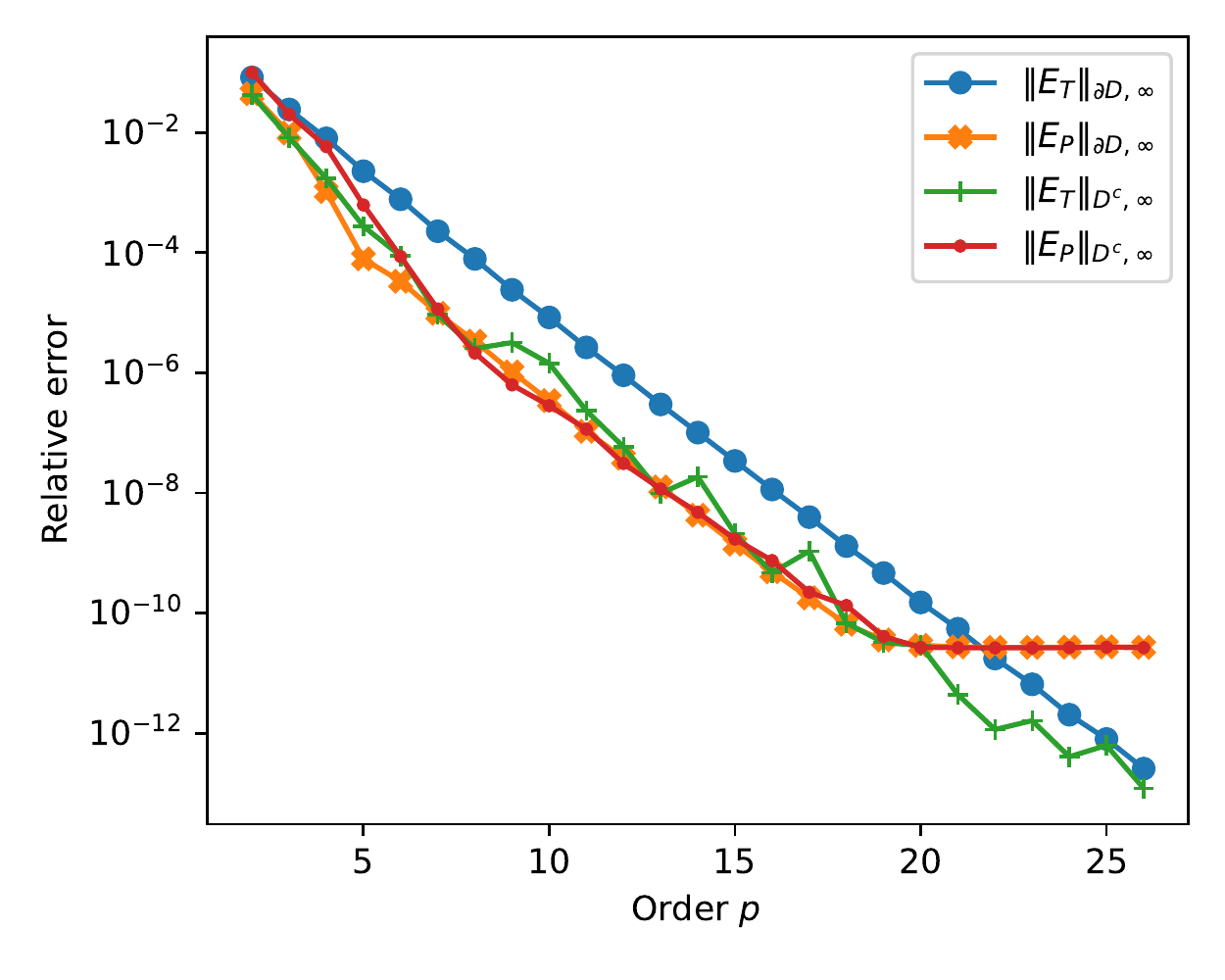}
  \includegraphics[width=0.45\textwidth]{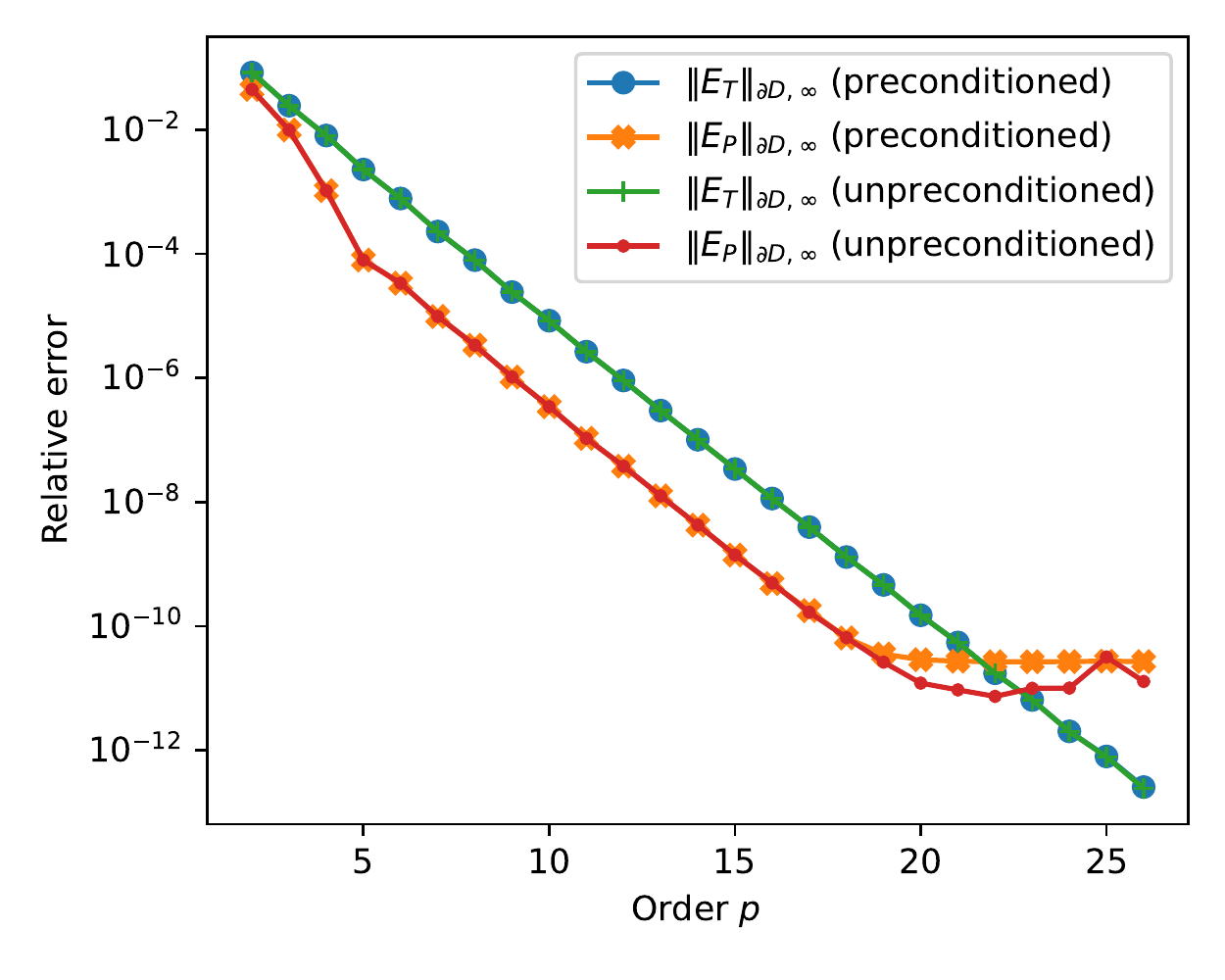}
    \caption{\(p\)-convergence using the coupled method. (Left: comparing the error
    norms measured over the boundary versus the volume. Right: comparing the errors
    with and without using the block preconditioner.)}
  \label{fig:conv1}
\end{figure}

We also perform \(h\)-convergence tests by fixing $p$ and shrinking panel sizes. For this test,
we set the QBX order to be $p+4$ and fix FMM order to be $15$. The goal for such choices is to
make the QBX truncation error and FMM error much smaller than the discretization error; therefore,
we should have convergence order of $p+1$ as $h\rightarrow 0$. The results are
shown in Figure \ref{fig:convh}, where the data confirms our expected convergence order.

\begin{figure}[ht]
  \centering
  \includegraphics[width=0.45\textwidth]{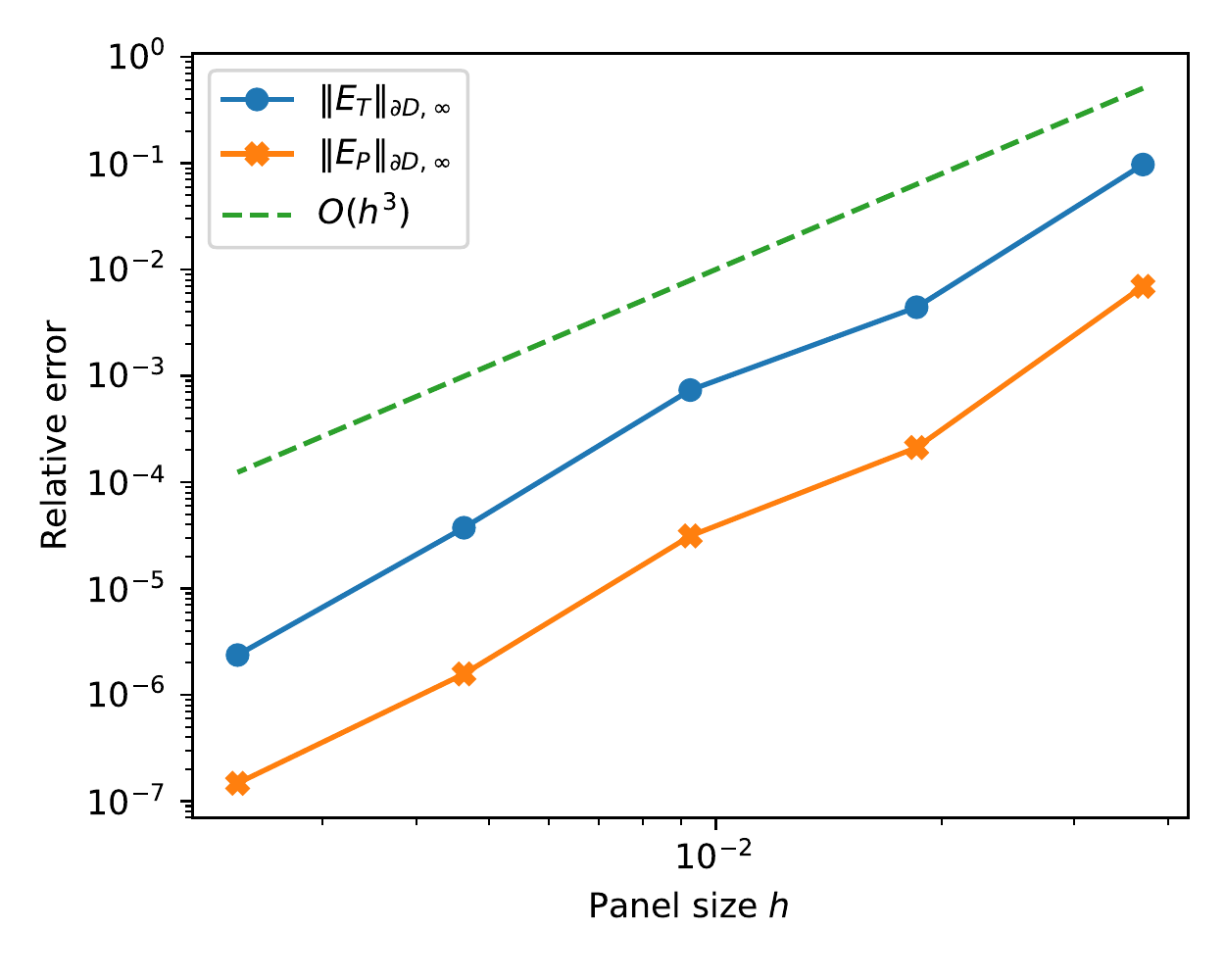}
  \includegraphics[width=0.45\textwidth]{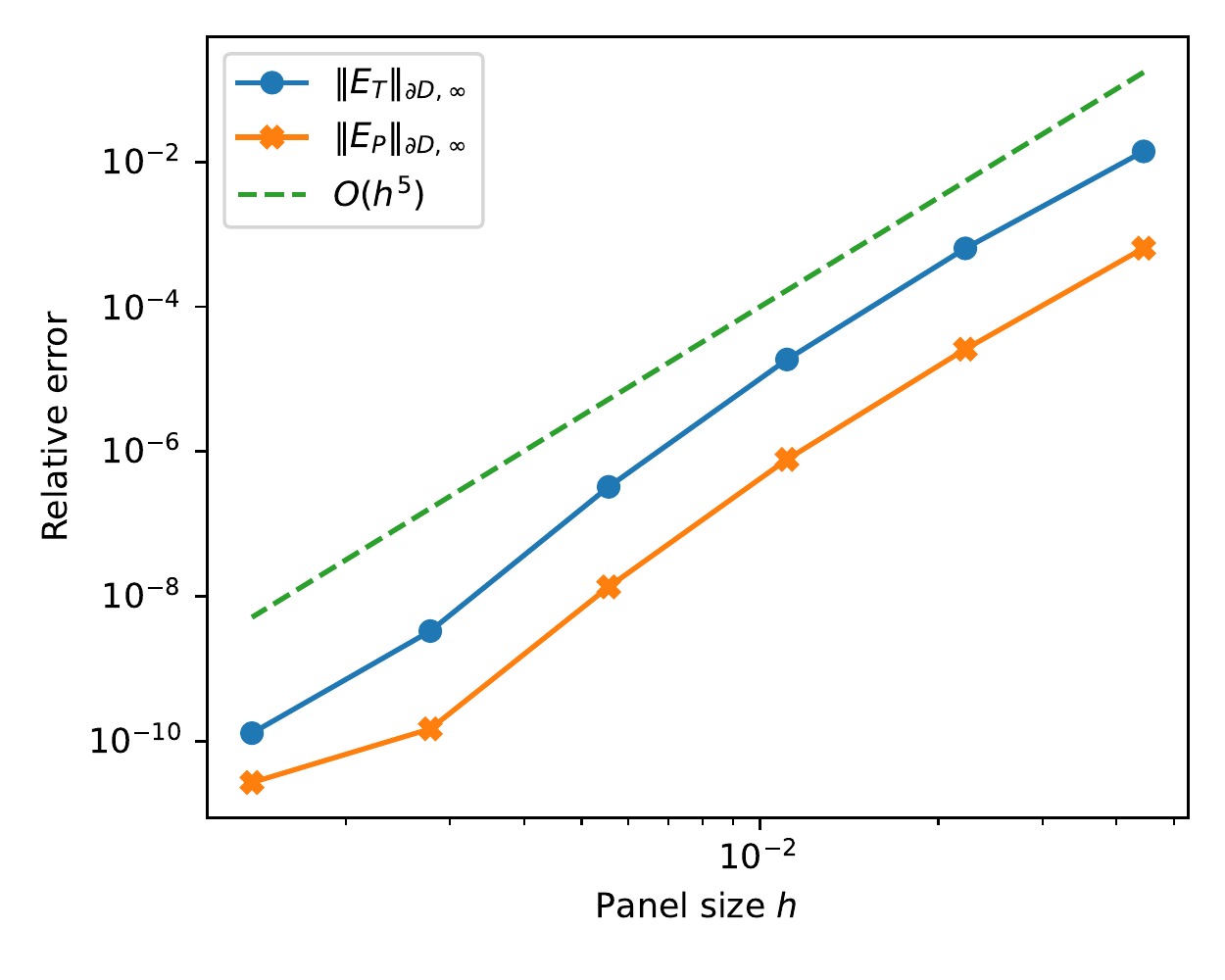}
    \caption{\(h\)-convergence using the coupled method. (Left: $p=2$. Right: $p=4$.)}
  \label{fig:convh}
\end{figure}

\subsection{The coupled method vs the decoupled method}

Using the same problem setup as in Section \ref{sec:res-coupled}, we compare the coupled
\reviewresponse{Reviewer 1, comment 2}
method and the decoupled method. We present the results in Figure \ref{fig:conv2}. For low order
cases, the decoupled method yields accuracy comparable to the coupled method; however,
as we increase resolution, the coupled method can achieve $10^{-12}$ accuracy, while
convergence for the decoupled method stalls at around $10^{-7}$. This phenomenon can be
explained by the large condition number of the decoupling transform,
\begin{equation*}
    \kappa\left(\begin{bmatrix} \Omega & t_+(1-i\gamma\Lambda) \\ \Omega & t_-(1-i\gamma\Lambda) \end{bmatrix}\right)
    \approx 4.19 \times 10^4.
\end{equation*}
For all three methods, the number of iterations needed does not increase as we
increase the resolution, which is as expected due to our formulations being
second-kind. Furthermore, the block preconditioner reduces the iteration count
significantly with little added cost. Roughly speaking, the decoupled method
requires similar number of GMRES iterations as the preconditioned coupled
method.

\begin{figure}[ht]
  \centering
  \includegraphics[width=0.45\textwidth]{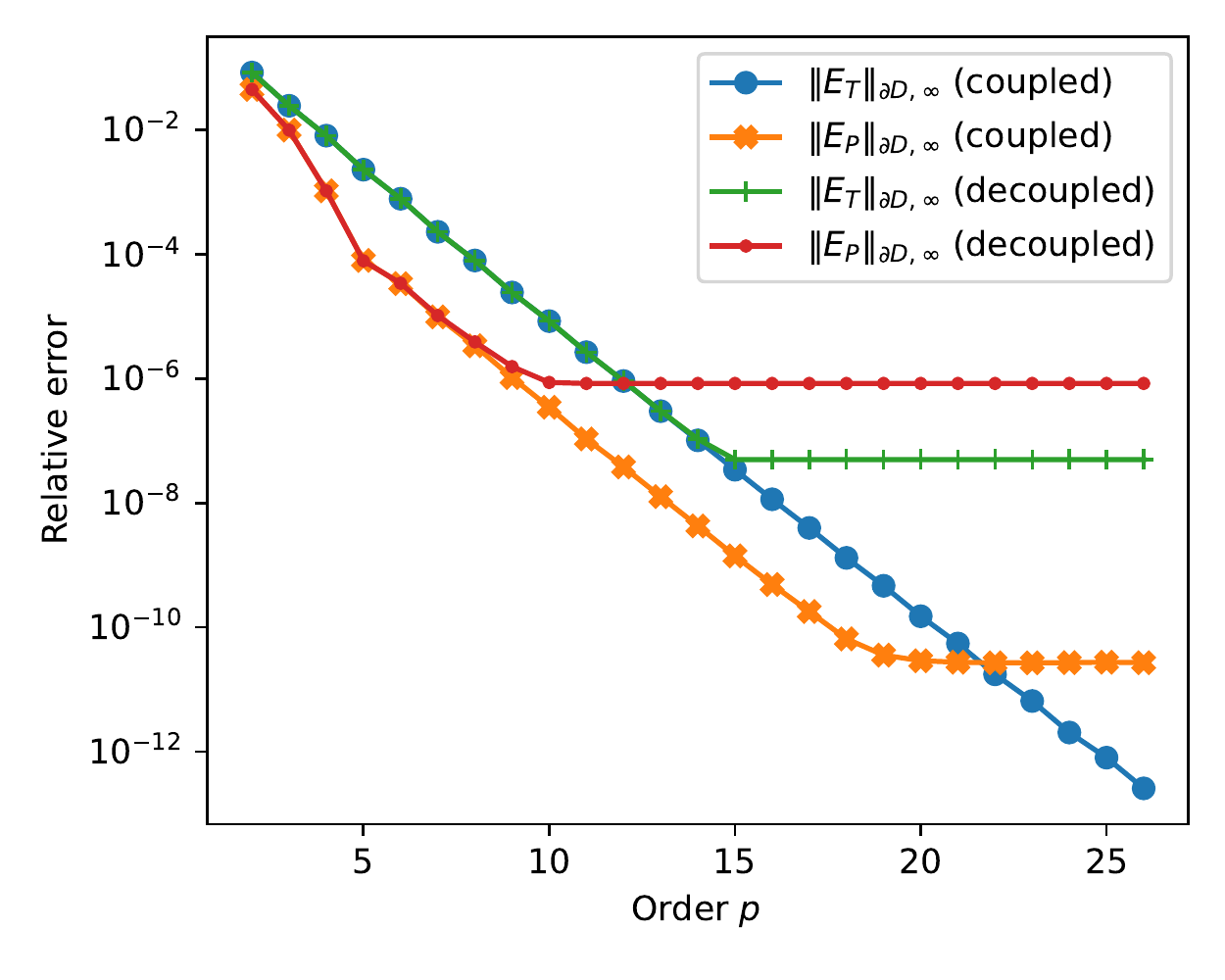}
  \includegraphics[width=0.45\textwidth]{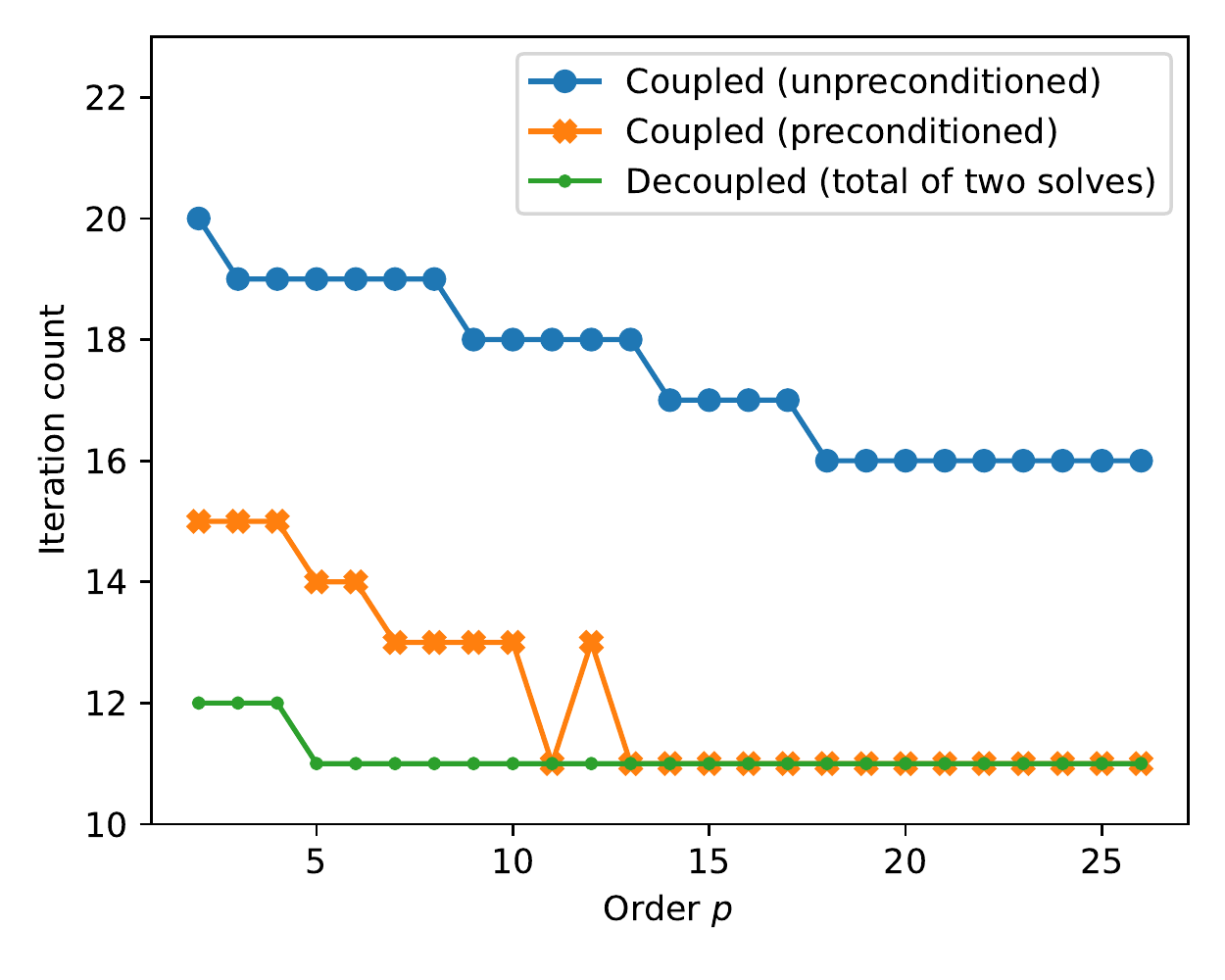}
    \caption{The coupled vs the decoupled method. (Left: convergence when increasing order.
    Right: total number of GMRES iterations.)}
  \label{fig:conv2}
\end{figure}

To further illustrate the effects of the left block preconditioner, we plot the
spectrum of the linear system with and without it in Figure \ref{fig:spectrum},
all other parameters fixed.
\reviewresponse{Reviewer 1, comment 1}
As expected, we see that the preconditioner causes the spectrum to cluster at
one location in the complex plane instead of two. This effect explains the
reduction factor of iteration counts.

\begin{figure}[ht]
  \centering
  \includegraphics[width=0.45\textwidth]{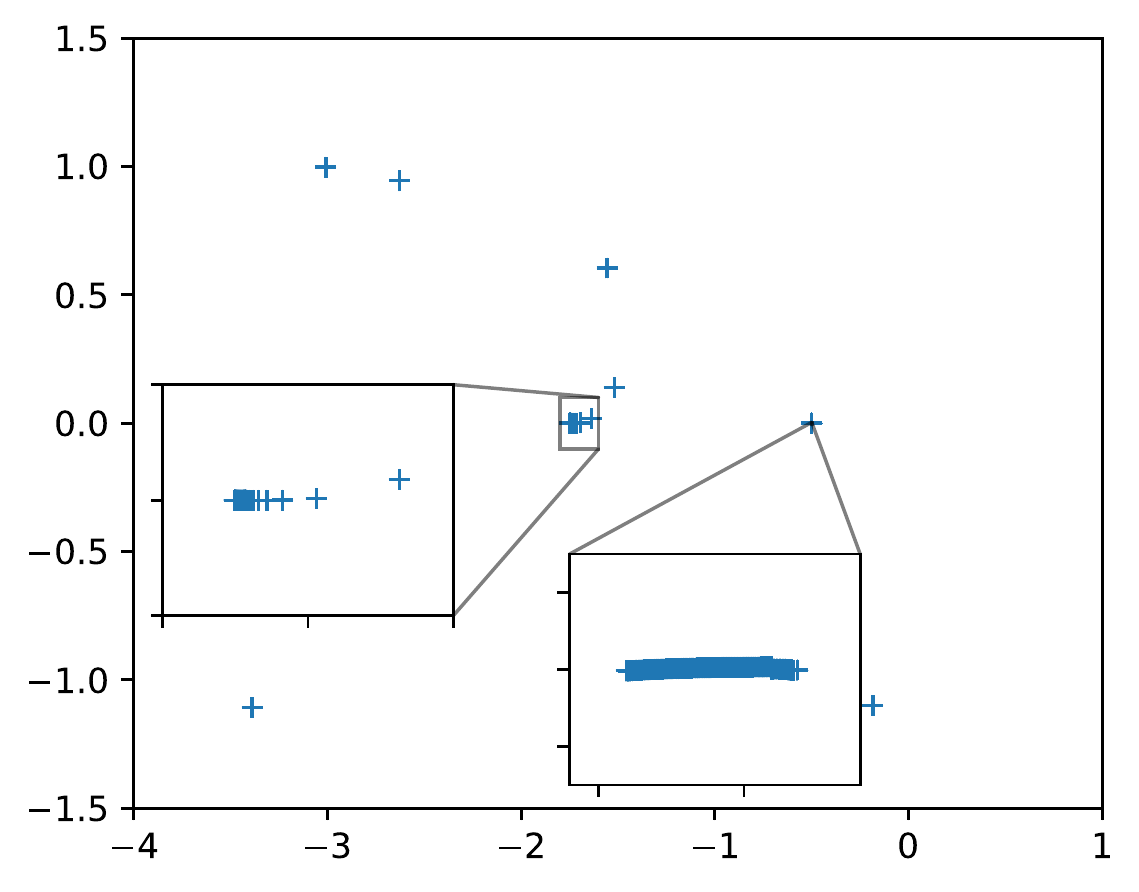}
  \includegraphics[width=0.47\textwidth]{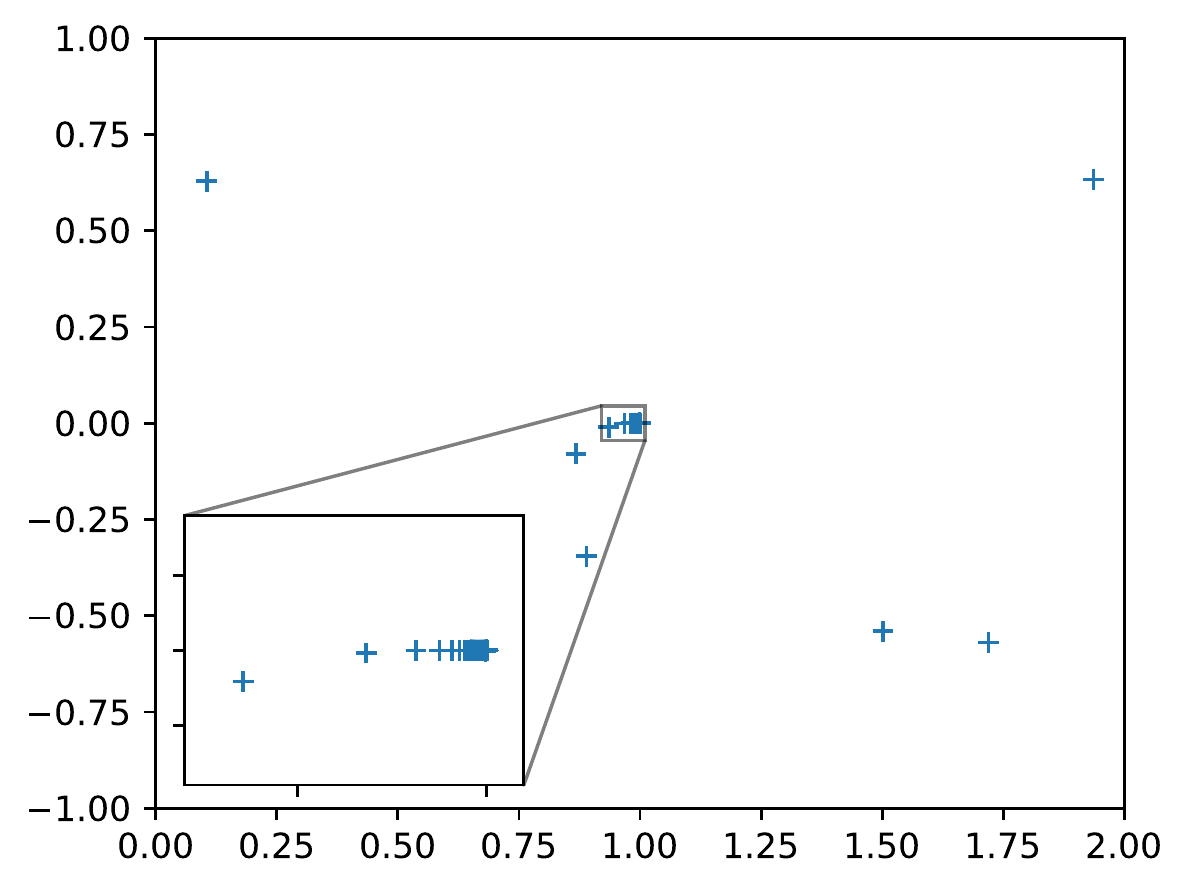}
    \caption{Spectrum plots of the coupled method over the complex plane where
    the \(x\) axis is the real part and \(y\) axis is the imaginary part.
    (Left: without left block preconditioning. Right: with left block preconditioning.)}
  \label{fig:spectrum}
\end{figure}

\subsection{The projection method for solution in the bulk region}

The projection method (\ref{eqn:projection-method}) is a variant of the decoupled method
for solution away from the boundary. To test how it performs, we
use it to find the solution over $[0, 10.5]^2 \mathbin{\backslash} D$, where
$D$ is as shown in the left plot of Figure \ref{fig:apple_project}. The red
dots marks the point sources used to produce the manufactured solution. The
lower-left half shows the relative error of the temperature field $E_T$, while
the upper-right half shows the relative error of the pressure field $E_P$. As
expected, the solution away from the boundary is accurate.

To see how the error decays with respect to the distance from boundary, in the
right plot of Figure \ref{fig:apple_project}, we plot the $l_\infty$ norms of
the relative errors over a set of equidistant curves against the curves'
distance from $\partial D$. We observe that the errors decay exponentially with
respect to the distance from boundary. For this example, the projected solution
is indistinguishable from the non-projected solution where the distance is
greater than $0.1$, which is only $1.6\%$ of the acoustic wave length.
Therefore, if the goal is to obtain solution in the bulk region, the projection
method can be very advantageous because it does not need to resolve the length
scale of the thermal boundary layer.

\begin{figure}[ht]
  \centering
  \includegraphics[width=0.45\textwidth]{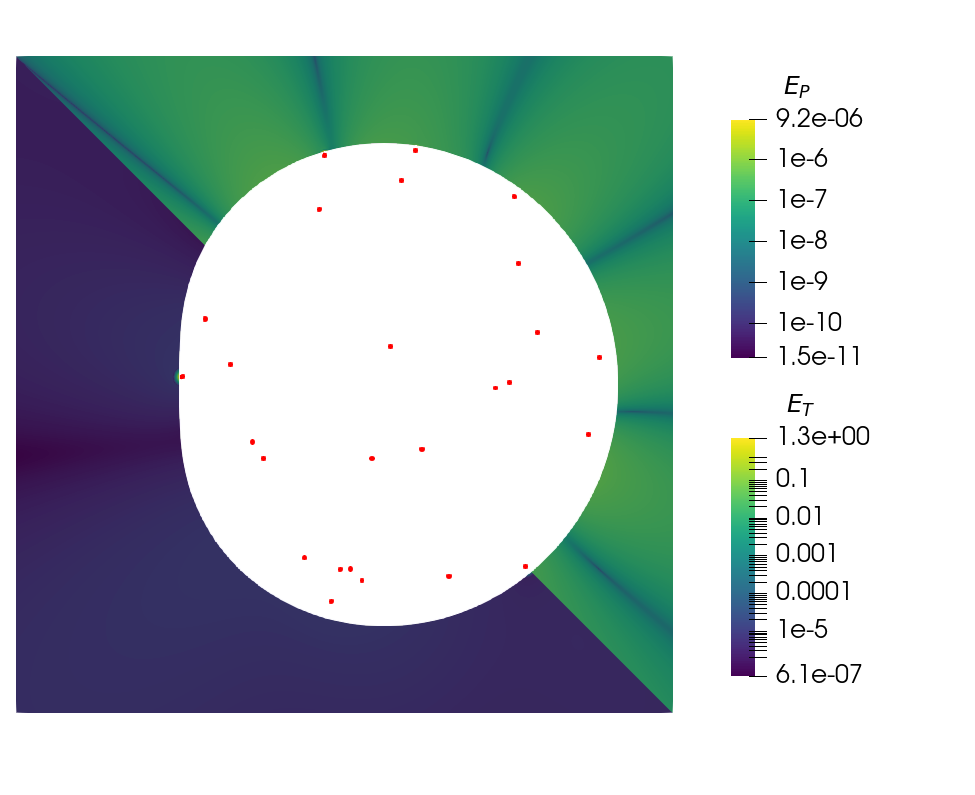}
  \includegraphics[width=0.45\textwidth]{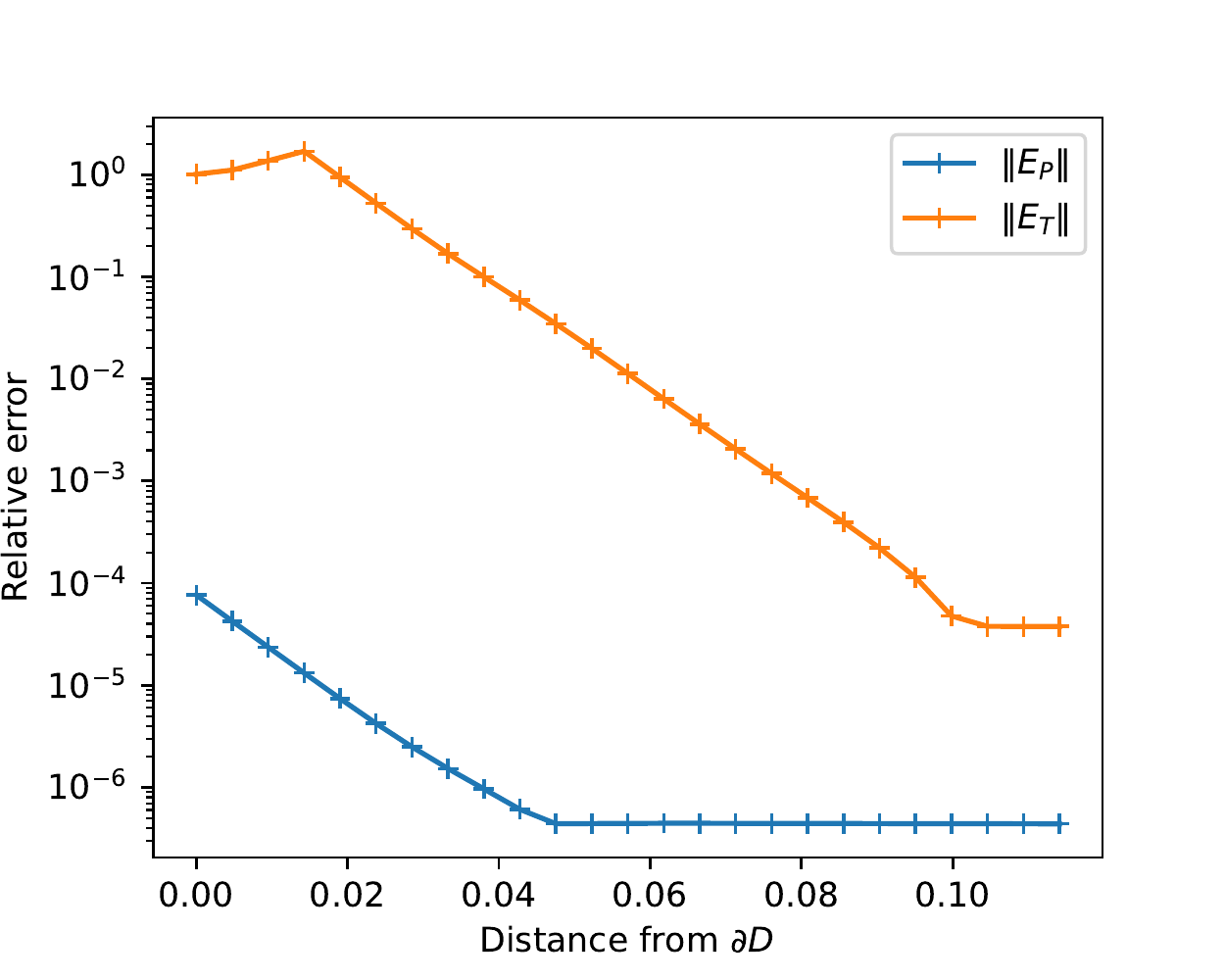}
    \caption{Error using the projection method. (Left: showing the geometry and
    relative error $E_T$ (lower-left), $E_P$ (upper-right). Right: showing that
    error decays exponentially w.r.t. the distance from the boundary.)}
  \label{fig:apple_project}
\end{figure}

It is worth pointing out that because $V_p$ is a linear combination of
temperature and pressure, by projecting the solution into only acoustic modes,
we are still solving a model with thermoacoustic coupling.

\subsection{Torus with the coupled method}
We apply the preconditioned coupled method to solve on a torus geometry with
major radius $0.7$ and minor radius $0.07$ as shown in Figure \ref{fig:torus}.
For the discretization, we set the mesh size $h=0.02$, discretization order $p=4$,
QBX order $10$, and FMM order $15$. Note that when computing layer potentials,
it is unusual to set the QBX order to be much higher than the
discretization order like this. The reason for the excessive QBX
order here is to resolve the boundary layer. The discrete system has roughly
$1.97$ million unknowns. The errors over the boundary are $\|
E_T\|_{\partial D} = 3.20\times 10^{-3}$, $\| E_P\|_{\partial D} =
1.74\times 10^{-3}$.  In Figure \ref{fig:torus}, part of the torus surface is
shown in a ``cut-off'' manner to reveal the source point that gives rise to the
manufactured solution. The rest of the torus surface shows a color map of the
residual.  And the plane surfaces behind the torus show a color map of
$\max(E_T, E_P)$ in the volume.

\begin{figure}[ht]
  \centering
  \includegraphics[width=0.6\textwidth]{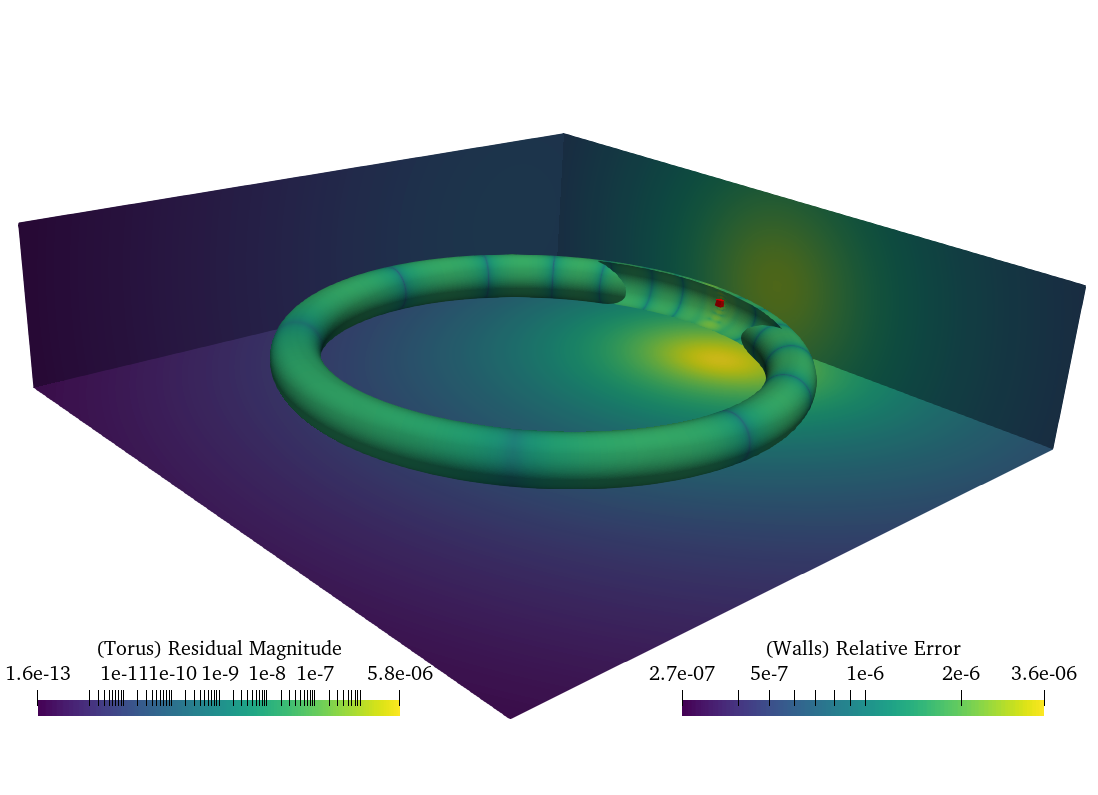}
  \caption{Torus with the coupled method.}
  \label{fig:torus}
\end{figure}

\subsection{A tuning fork geometry with the projection method}
When only the solution away from the boundary is needed, the projection method
allows solving for large problems that are otherwise too expensive to be solved
with the coupled method by circumventing the need to resolve the boundary
layer.  In this test, we solve over a tuning fork geometry as shown in Figure
\ref{fig:qtf} using the projection method. The size of the tuning fork is roughly
$1.5\times0.34\times6.23$. Since there is no boundary layer in the acoustic modes,
we use a QBX order of $2$ and FMM order $15$. The discrete system has roughly
$2.7$ million unknowns. In Figure \ref{fig:qtf}, part of the tuning fork surface is
made transparent to reveal the source points that give rise to the manufactured solution.
The remaining part of the tuning fork is colored by the residual magnitude. Also,
the plane surfaces behind the tuning fork show a color map of $\max(E_T, E_P)$
in the volume.
Although the error is large on the boundary surface ($\| E_T\|_{\partial
D} = 1.00$, $\| E_P\|_{\partial D} = 0.01$), the solution in the bulk
region has five digits of accuracy, as shown in Figure \ref{fig:qtf-err}.

\begin{figure}[ht]
  \centering
  \includegraphics[width=0.75\textwidth]{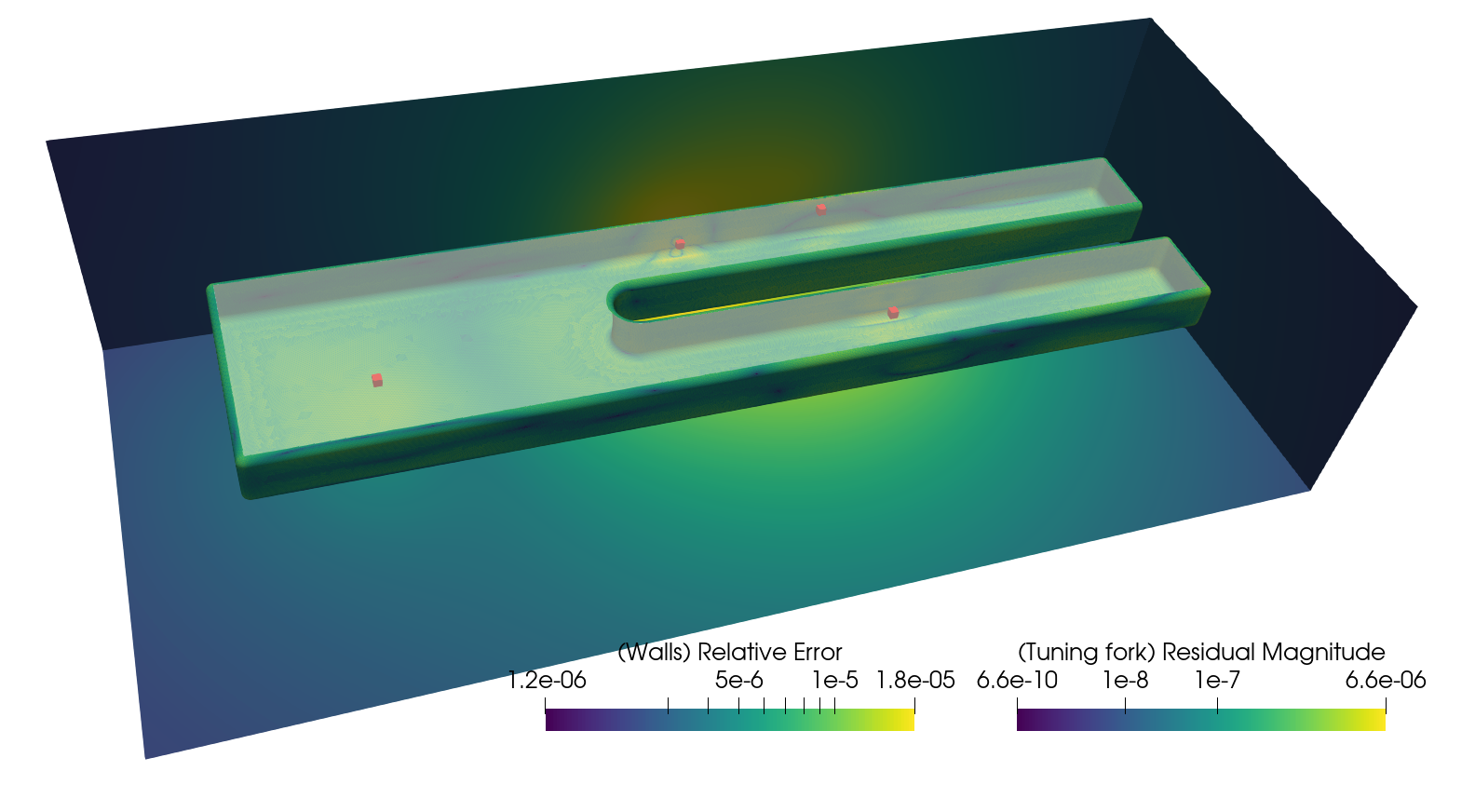}
  \caption{Tuning fork with the projection method.}
  \label{fig:qtf}
\end{figure}

\begin{figure}[ht]
  \centering
  \includegraphics[width=0.5\textwidth]{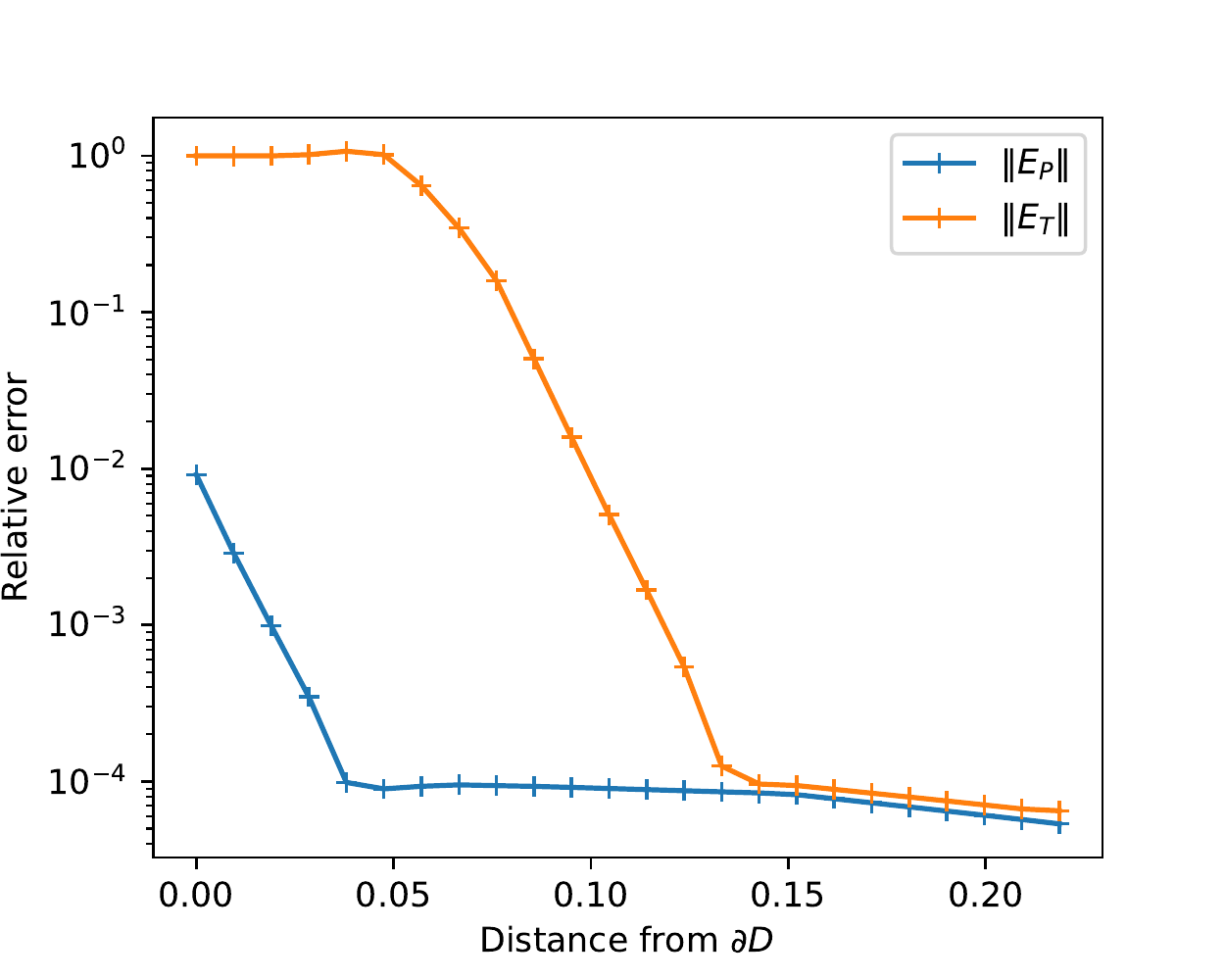}
  \caption{Error of the tuning fork solution decays with increasing distance from the boundary.}
  \label{fig:qtf-err}
\end{figure}

\section{Conclusion}
\label{sec:conclusion}

In this work, we have derived the free-space Green's function and the decoupled
PDE form of the Morse-Ingard equations. Our results are an extension of the
analysis in \cite{kaderli_analytic_2017}. Using the obtained relation between
the Morse-Ingard equations and the Helmholtz equations, we have provided
analogs of the Sommerfeld radiation conditions and Green's representation
formulae for the Morse-Ingard equations.  Based on our analysis, we have
developed three integral equations methods for the Neumann problem of the
Morse-Ingard equations in the exterior domain: the (preconditioned) coupled
method, the decoupled method, and the projection method. All three methods are
based on SKIE formulations, and can be solved with linear complexity in the number of
discretization nodes, as demonstrated in our numerical experiments using the
GIGAQBX algorithm.

Through numerical tests using realistic parameter values, we have shown that
for all three methods, the number of GMRES iterations does not grow with the
problem size. Of the three methods, the coupled method is the best-conditioned and can achieve more than $11$ digits of accuracy with
sufficient resolution.  In comparison, the decoupled method requires less
computational cost, but it is prone to the ill-conditioning of the decoupling
transform, which depends on the equation parameters. The projection method represents a
trick applied to the decoupled method that can be used to circumvent the need
to resolve the boundary layer when only the solution away from the boundary
is desired. This is similar to the workaround used in \cite{wei_integral_2020}
for long-time simulations.

It has been repeatedly observed that layer potentials with boundary layers are
resolution-hungry and pose challenges in developing more efficient
integral-equation-based solvers for such problems. For future research, we
will look to improve the QBX method to more efficiently resolve boundary
layers.

\section{Acknowledgments}
\label{sec:ack}

We would like to thank Professor John Zweck and Doctor Artur Safin for useful
discussions on the model and boundary conditions.  Our thanks also goes to
Jacob Davis for sharing preliminary derivations related to the Green’s
function.

The first two authors' research was supported by the National Science Foundation
under award numbers SHF-1911019 and OAC-1931577 as well as by the
Department of Computer Science at the University of Illinois at
Urbana-Champaign.  The third author was supported by the National Science Foundation under award numbers SHF-1909176.

\section*{References}
\label{sec:org0b55d8a}
\bibliography{refs}
\end{document}